\documentclass{amsart}
\usepackage{setspace}
\usepackage{a4}
\usepackage{amsthm,latexsym}
\usepackage{amsfonts}
\usepackage{graphicx}
\usepackage{textcomp}
\usepackage{cite}
\usepackage{enumerate}
\usepackage[mathscr]{euscript}
\usepackage{mathtools}
\newtheorem{theorem}{Theorem}[section]

\newtheorem{conjecture}[theorem]{Conjecture}

\newtheorem{definition}[theorem]{Definition}

\title{This is the title}
\usepackage{amssymb}
\usepackage{amsmath}
\usepackage{tikz}
\usepackage{hyperref}
\usepackage{mathtools}
\usetikzlibrary{cd}
\usepackage{fancyhdr}

\begin{document}
	\hrule\hrule\hrule\hrule\hrule
	\vspace{0.3cm}	
	\begin{center}
		{\bf{LOCALIZED BOUNDED BELOW APPROXIMATE SCHAUDER FRAMES ARE FINITE UNIONS OF APPROXIMATE RIESZ SEQUENCES}}\\
		\vspace{0.3cm}
		\hrule\hrule\hrule\hrule\hrule
		\vspace{0.3cm}
		\textbf{K. MAHESH KRISHNA}\\
		Post Doctoral Fellow \\
		Statistics and Mathematics Unit\\
		Indian Statistical Institute, Bangalore Centre\\
		Karnataka 560 059, India\\
		Email: kmaheshak@gmail.com\\
		
		Date: \today
	\end{center}


\hrule
\vspace{0.5cm}
\textbf{Abstract}: Based on the truth of Feichtinger conjecture by Marcus, Spielman and Srivastava \textit{[Ann. of Math. (2), 2015]} and from the localized version by Gr\"{o}chenig  \textit{[Adv. Comput. Math., 2003]}, we introduce the notion of localization of approximate Schauder frames (ASFs) and approximate Riesz sequences (ARSs). We show that localized bounded below ASFs are finite unions of ARSs.


\vspace{0.5cm}
\hrule 

\section{Introduction}
In the beginning years of $21^{th}$ century, Prof. Feichtinger  formulated the following conjecture based on his extensive work on Gabor/Weyl-Heisenberg frames (see \cite{CHRISTENSEN2014} for the history and  \cite{STOEVA, DUFFINSCHAEFFER, CHRISTENSEN, FEICHTINGERSTROHMER, FEICHTINGERSTROHMER2, HEIL, FEICHTINGERGROCHENIG, GROCHENIG2007, GROCHENIG2014, BENEDETTOWALNUT, GROCHENIGBOOK, GROCHENIG2015} for general theory).
\begin{conjecture}\cite{CASAZZACHRISTENSENLINDNERVERSHYNIN} \textbf{(Feichtinger Conjecture/Marcus-Spielman-Srivastava Theorem)}\label{FCON}
	\textbf{Let  $\{\tau_n\}_n$ be  a frame for  a Hilbert space  $\mathcal{H}$ such that 
		\begin{align*}
			0<\inf_{n\in \mathbb{N}}\|\tau_n\|.
		\end{align*}
		Then $\{\tau_n\}_n$ can be partitioned into a finite union of Riesz sequences for $\mathcal{H}$.}
\end{conjecture}
First breakthrough which supported  Conjecture \ref{FCON} occurred when Gr\"{o}chenig  proved it for intrinsically  localized frames \cite{GROCHENIG}. Shortly afterwords, it has been verified for certain classes of $\ell^1$-self-localized frames, wavelet frames, Gabor frames, frames of translates, frames formed by reproducing kernels  and exponential frames/frames of exponentials \cite{BOWNIKSPEEGLE, BALANCASAZZAHEILLANDAU, SPEEGLE2008, LATAPAULSEN, BARANOVDYAKONOV, LAWTON, WEBER}.  Conjecture \ref{FCON} received great attention after establishing its   equivalence  with Kadison-Singer conjecture \cite{CASAZZAFICKUSTREMAINWEBER, CASAZZAEDIDIN, CASAZZAKUTYNIOKSPEEGLETREMAIN}. Finally, the Feichtinger conjecture has been solved fully by  resolving Weaver's conjecture by Marcus, Spielman, and Srivastava in 2013 \cite{MARCUSSPIELMANSRIVASTAVA, WEAVER, BOWNIK, TIMOTIN, MARCUSSRIVASTAVA}. In this paper, we formulate a Banach space version of Conjecture \ref{FCON} and prove it for bounded below intrinsically localized ASFs (Theorem \ref{THM}).

\section{Localized bounded below ASFs are finite unions of  ARSs}
We consider the following most general notion of  approximate Schauder frames. In the entire paper,  $\mathcal{X}$ is a  separable Banach space and $\mathcal{X}^*$ is its dual. 

\begin{definition}\cite{CASAZZAHANLARSON, CASAZZADILWORTHODELL, FREEMANODELLSCHIUMPRECHT}
	Let $ \{\tau_n \}_{n}$ be a collection in $\mathcal{X}$ and $ \{f_n \}_{n}$ be a collection in $\mathcal{X}^*$.	The pair  $ (\{f_n \}_{n}, \{\tau_n \}_{n}) $  	is said to be an \textbf{approximate Schauder frame} (we write ASF)  for $\mathcal{X}$ if 
the \textbf{frame operator} 
\begin{align*}
S_{f, \tau}:\mathcal{X}\ni x \mapsto S_{f, \tau}x\coloneqq \sum_{n=1}^\infty f_n(x)\tau_n \in\mathcal{X}	
\end{align*}
 is a well-defined bounded linear invertible operator. 
		
\end{definition}
We use the following notion of `bounded below' for ASFs.
\begin{definition}
	An ASF $ (\{f_n \}_{n}, \{\tau_n \}_{n}) $  for $\mathcal{X}$ is said to be \textbf{bounded below} if 
	\begin{align*}
		\inf_{n \in \mathbb{N}}|f_n(\tau_n)|>0.
	\end{align*}
\end{definition}
Motivated from the definition of localization of frames \cite{GROCHENIG2004, FRNASIERGROCHENIG}, we introduce the following notion.
\begin{definition}
An ASF $ (\{f_n \}_{n}, \{\tau_n \}_{n}) $  for $\mathcal{X}$ is said to be \textbf{intrinsically/self localized} if 	there exist $s>1$ and $A>0$ such that 
\begin{align*}
|f_n(\tau_m)|\leq \frac{A}{(1+|n-m|)^s}, \quad \forall n, m \in \mathbb{N}.
\end{align*}
\end{definition}
Two notions of Riesz sequences for Banach spaces exist in literature, see 
\cite{CHRISTENSENSTOEVA, ALDROUBISUNWAI} and \cite{KRISHNAJOHNSON2}. Here we define another.
\begin{definition}\label{SRB}
An ASF $ (\{f_n \}_{n}, \{\tau_n \}_{n}) $  for $\mathcal{X}$ is  said to be an \textbf{approximate Riesz sequence} (we write ARS) if there exists a finite partition $Q_1, \dots, Q_N$ of $\mathbb{N}$ such that 
\begin{align*}
	\mathbb{N}=\bigcup_{j=1}^NQ_j
\end{align*}
and for each $1\leq j \leq N$, 
\begin{align*}
	\inf_{n \in Q_j}\left(|f_n(\tau_n)|-\sum_{m \in Q_j, m \neq n}|f_n(\tau_m)|\right)>0.
\end{align*}
\end{definition}
Note that for Hilbert spaces, if $\{f_n \}_{n}$ is determined by $\{\tau_n \}_{n}$ (Riesz representation), then Definition \ref{SRB} is equivalent (due to positivity) to the definition of Riesz sequence (see\cite{GROCHENIG}). We now formulate the following conjecture (some other are formulated in \cite{KRISHNA}).

\begin{conjecture}\label{CONJECTURE}
	\textbf{Every bounded below ASF can be partitioned  as a finite union of ARBs.}
\end{conjecture}
We now prove Conjecture \ref{CONJECTURE} for intrinsically localized ASFs with the help of following result.
\begin{theorem}\cite{GROCHENIG}\label{GT}
	\begin{enumerate}[\upshape(i)]
		\item For every $s>1$, 
		\begin{align*}
			D_s\coloneqq 	\sup_{x \in \mathbb{R}} \sum_{n=1} ^\infty \frac{1}{(1+|n-x|)^s}<\infty.
		\end{align*}
	\item  For every $s>1$, there exists a $C_s>0$ (which does not depend on  $\delta$) such that 	
	\begin{align*}
		\sup_{m \in \mathbb{N}} \sum_{n \in \mathbb{N}, n \neq m} \frac{1}{(1+|n-m|)^s}\leq \frac{C_s}{\delta^s}, 
	\end{align*}
whenever  
\begin{align*}
	\inf_{n, m \in \mathbb{N}, n \neq m}|n-m|\geq \delta.
\end{align*}
	\end{enumerate}
\end{theorem}
\begin{theorem}\label{THM}
Conjecture \ref{CONJECTURE} holds for intrinsically localized bounded below ASFs.	
\end{theorem}
\begin{proof}
Our proof is highly motivated from \cite{GROCHENIG}. Let $ (\{f_n \}_{n}, \{\tau_n \}_{n}) $  be an intrinsically localized bounded below ASF for  $\mathcal{X}$.  
Define 
\begin{align*}
C\coloneqq \inf_{n \in \mathbb{N}}|f_n(\tau_n)|>0.
\end{align*}	
Since $ (\{f_n \}_{n}, \{\tau_n \}_{n}) $  intrinsically localized	there exist $s>1$ and $A>0$ such that 
\begin{align*}
	|f_n(\tau_m)|\leq \frac{A}{(1+|n-m|)^s}, \quad \forall n, m \in \mathbb{N}.
\end{align*}
Let $C_s$ be the constant  as in Theorem \ref{GT}. Choose natural number $M$ such that 
\begin{align*}
	\frac{AC_s}{M^s}\leq \frac{C}{2}.
\end{align*}
We now partition $\mathbb{N}$ into $Q_1, \dots, Q_N$ such that 
\begin{align*}
	\mathbb{N}\coloneqq \bigcup_{j=1}^NQ_j
\end{align*}
and for each $ 1\leq j\leq N$,
\begin{align*}
	\inf_{n, m \in Q_j, n \neq m}|n-m|\geq M.	
\end{align*}
(Note that there are infinitely many partitions of $\mathbb{N}$ of this type.) Let $ 1\leq j\leq N$. Then using (ii) in Theorem \ref{GT}
\begin{align*}
	\sup_{n \in Q_j}\sum_{m \in Q_j, m \neq n}|f_n(\tau_m)|\leq A \sup_{n \in Q_j}\sum_{m \in Q_j, m \neq n}\frac{1}{(1+|n-m|)^s}\leq A \frac{C_s}{M^s}\leq \frac{C}{2}.
\end{align*}
Therefore for each fixed $ 1\leq j\leq N$
\begin{align*}
\inf_{n \in Q_j}\left(|f_n(\tau_n)|-\sum_{m \in Q_j, m \neq n}|f_n(\tau_m)|\right)&=\inf_{n \in Q_j}|f_n(\tau_n)|-\sup_{n \in Q_j}\sum_{m \in Q_j, m \neq n}|f_n(\tau_m)|\\
&\geq C-\frac{C}{2}=\frac{C}{2}>0.
\end{align*}
Since $j$ was arbitrary, we get the theorem.
\end{proof}

  \bibliographystyle{plain}
 \bibliography{reference.bib}

\end{document}